\documentclass{article}
\usepackage{mathrsfs}
\usepackage{graphicx}
\usepackage{bbm}
\usepackage{relsize}
\usepackage{enumerate}
\usepackage{amssymb,amsmath,amsthm}
\usepackage{bm}
\usepackage[title]{appendix}
\usepackage{url}

\usepackage{titlesec}
\usepackage{verbatim}
\usepackage{mathrsfs}
\usepackage[usenames, dvipsnames]{color}
\usepackage[letterpaper, portrait, margin=1in]{geometry}
\usepackage{amsmath, amssymb, amsthm}
\usepackage{graphicx}
\usepackage{mathrsfs}
\usepackage{xcolor}
\usepackage{bbm}
\usepackage{mathtools}
\usepackage[linesnumbered,lined,ruled]{algorithm2e}
\usepackage{algpseudocode}
\usepackage{chngcntr}
\usepackage{float,stmaryrd}

\theoremstyle{definition} 
\newtheorem{theorem}{Theorem}[section]
\newtheorem{corollary}[theorem]{Corollary}
\newtheorem{lemma}[theorem]{Lemma}
\newtheorem{proposition}[theorem]{Proposition}

\newtheorem{example}[theorem]{Example}

\newtheorem{question}[theorem]{Question}

\DeclarePairedDelimiter\floor{\lfloor}{\rfloor}
\DeclarePairedDelimiter\ceil{\lceil}{\rceil}

\newcommand{\cart}{\boxempty}

\DeclareMathOperator{\val}{val}

\DeclareMathOperator{\gon}{gon}
\DeclareMathOperator{\mfgon}{mfgon}

\title{Multiplicity-Free Gonality on Graphs}
\author{Frances Dean, Max Everett, and Ralph Morrison}
\date{}

\begin{document}

\maketitle

\begin{abstract}
The divisorial gonality of a graph is the minimum degree of a positive rank divisor on that graph.  We introduce the multiplicity-free gonality of a graph, which restricts our consideration to divisors that place at most \(1\) chip on each vertex.  We give a sufficient condition in terms of vertex-connectivity for these two versions of gonality to be equal; and we show that no function of gonality can bound multiplicity-free gonality, even for simple graphs.  We also prove that multiplicity-free gonality is NP-hard to compute, while still determining it for graph families for which gonality is currently unknown.  We also present new gonalities, such as for the wheel graphs.
\end{abstract}

\section{Introduction}
Chip-firing games on graphs provide a discrete, combinatorial analog to divisor theory on algebraic curves.  The theory of divisors on graphs mirrors that on algebraic curves through analogs of such results as the Riemann-Roch theorem \cite{bn} and results on graphs imply results on curves through Baker's specialization lemma \cite{baker}.  This allows for purely combinatorial methods to prove theorems in algebraic geometry.  This theory has been developed for both metric graphs and finite (non-metric) graphs.  Throughout this paper we work with finite multigraphs, with multiple edges allowed between vertices but no loops from a vertex to itself (if there are no multiple edges between any pair of vertices, the graph is called \emph{simple}).

One much-studied invariant of curves, and more recently of graphs, is \emph{gonality} (specified as divisorial gonality for graphs). In either setting, it can be defined as the minimum degree of a positive rank divisor; in the algebro-geometric world it is also the minimum degree of a surjective morphism from the curve to a projective line.   In the graph theoretic world, gonality admits a description in terms of a game:  Player A places \(k\) chips on the vertices of a graph, and Player B adds \(-1\) chips.  If Player A can perform certain ``chip-firing'' moves on the graph to eliminate debt, then Player A wins; if not, then Player B wins.  The gonality of the graph can be defined as the minimum \(k\) such that Player A has a placement of \(k\) chips that wins against Player B, no matter how Player B plays.     We note that there are a number of other definitions of graph gonality inequivalent to divisorial gonality, including stable gonality  and stable divisorial gonality \cite{stable_gonality}.  Throughout this paper ``gonality'' without a qualifier refers to divisorial gonality.

We introduce a variation of gonality which we call \emph{multiplicity-free gonality}.  Loosely speaking, a divisor is mutiplicity-free if it places either \(0\) or \(1\) chips on each vertex.  The multiplicity-free gonality of the graph is then the minimum degree of a multiplicity-free divisor that wins the gonality game.  It immediately follows that gonality is at most multiplicity-free gonality; we will see in Section \ref{section:background} that there are graphs with multiplicity-free gonality strictly larger than gonality, such as the \emph{slashed ladder graph} in Figure~\ref{figure:no_multiplicity_free}.

\begin{figure}[hbt]
   		 \centering
\includegraphics[scale=1]{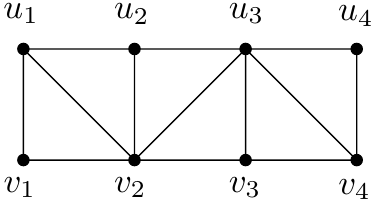}
	\caption{The \(2\times 4\) slashed ladder graph, with gonality strictly smaller than multiplicity-free gonality}
	\label{figure:no_multiplicity_free}
\end{figure}

One reason for introducing multiplicity-free gonality is that it is in many ways more feasible to study than traditional gonality.  Although both are NP-hard to compute (see \cite{gij} for gonality, and our Theorem \ref{theorem:mfgon_nphard} for multiplicity-free gonality), brute-force methods have to consider significantly fewer divisors for multiplicity-free gonality.  Moreover, multiplicity-free gonality is much more amenable to proofs utilizing such techniques as Dhar's burning algorithm \cite{dhar}; in Section \ref{section:families} we will leverage this to compute the multiplicity-free gonality of an arbitrary \(\ell\)-dimensional rook's graph, along with other graph families.

The question then becomes when are gonality and multiplicity-free gonality equal.  We prove in Section \ref{section:equality} that if a simple graph has gonality equal to its vertex connectivity, then it also has gonality equal to multiplicity-free gonality.  We also provide several negative results that show the limitations of relating gonality with multiplicity-free gonality.  In Section \ref{section:inequality} we prove that for graphs of any fixed gonality \(2\) or more, multiplicity-free gonality can take on \emph{any} larger value, meaning that we cannot bound multiplicity-free gonality with a function of gonality; except when the fixed gonality is \(2\), we can achieve the same result even for simple graphs.  We also show in Section \ref{section:families} that for any \(r\geq 4\), there exists an \(r\)-regular graph with gonality strictly smaller than multiplicity-free gonality; that section also includes our results on particular graph families, and our proof that multiplicity-free gonality is NP-hard to compute.

\medskip

\noindent \textbf{Acknowledgements.}  The authors were supported by Williams College, by the 2018 and 2020 iterations of the SMALL REU, and by NSF Grants DMS-1659037 and DMS-2011743.  They are grateful to Ivan Aidun, Teresa Yu, and Julie Yuan for many conversations on wheels, antiprisms, and other graphs; to Josh Carlson for reading through a preliminary version of this paper; and to Lisa Cenek, Lizzie Ferguson, Eyobel Gebre, Cassandra Marcussen, Jason Meintjes, Liz Ostermeyer, and Shefali Ramakrishna for help in gonality computation.

\section{Background and initial results}
\label{section:background}

A graph \(G=(V,E)\) is a finite collection of vertices \(V(G)\) with a finite multiset of edges \(E(G)\) connecting distinct vertices of \(V(G)\).  A graph is \emph{connected} if it is possible to travel from every vertex to every other vertex along edges. A set \(S\subset V(G)\) of vertices is called a \emph{vertex-cut} if deleting the vertices in \(S\) from \(G\) yields either a disconnected graph, or a graph on one vertex; the minimum cardinality of a vertex-cut of \(G\) is called the \emph{vertex-connectivity} \(\kappa(G)\) of \(G\).  The number of edges incident to a vertex \(v\) is called the \emph{valence}\footnote{Often this is called the \emph{degree} of \(v\); in this paper we reserve the word ``degree'' for another meaning.} of \(v\), denoted \(\val(v)\).  Given a subset \(U\subset V(G)\), the \emph{outdegree} of \(U\) is the number of edges with one endpoint in \(U\) and the other endpoint in \(U^C\).

Letting \(G\) be a connected graph, we let \(\textrm{Div}(G)\) denote the finite abelian group on the vertices of \(G\);  as a group,  \(\textrm{Div}(G)\)  is isomorphic to \(\mathbb{Z}^{|V(G)|}\).  Any element of \(\textrm{Div}(G)\)  is called a \emph{divisor}.  We write \(D\in \textrm{Div}(G)\) as
\[D=\sum_{v\in V(G)}a_v\cdot (v),\]
where \(a_v\in\mathbb{Z}\).  The coefficient of \((v)\) in \(D\) is sometimes denoted \(D(v)\); that is, \(D(v)=a_v\).  The \emph{degree} of a divisor is the sum of the coefficients:
\[\deg(D)=\sum_{v\in V(G)}D(v)=\sum_{v\in V(G)}a_v.\]
We say \(D\) is \emph{effective}, written \(D\geq 0\), if \(D(v)\geq 0\) for all \(v\).

The \emph{Laplacian matrix} \(\mathcal{L}\) of \(G\) is the \(|V(G)|\times |V(G)|\) matrix whose diagonal entries record the valences of the vertices of \(G\), and whose off-diagonal entry \(\mathcal{L}_{ij}\) is equal to minus the number of edges connecting vertex \(i\) to vertex \(j\).  We say that two divisors \(D,D'\in \textrm{Div}(G)\) are \emph{equivalent}, written \(D\sim D'\) if \(D-D'\) (thought of as a vector) is in the column space of \(\mathcal{L}\).  This forms an equivalence relation on \(\textrm{Div}(G)\).

This equivalence relation admits a more intuitive description in the language of chip-firing games.  We think of a divisor \(D\) as a placement of poker chips on a graph, where vertex \(v\) has \(D(v)\) chips; note that a vertex may have a negative number of chips, in which case we describe that vertex as being ``in debt.''  We can then perform chip-firing moves.  The \emph{chip-firing move at \(v\)} transforms \(D\) to \(D'\) by removing \(\textrm{val}(v)\) chips from \(v\) and moving them along each edge incident to \(v\) to its neighbors.
Then two divisors \(D\) and \(D'\) are equivalent under our Laplacian definition if and only if they differ by a sequence of chip-firing moves.  Three equivalent divisors are illustrated in Figure \ref{figure:no_multiplicity_free_chips1}; the second is obtained from the first by chip-firing \(v_1\), and the third is obtained from the second by chip-firing \(u_1\).

\begin{figure}[hbt]
    \centering
    \includegraphics{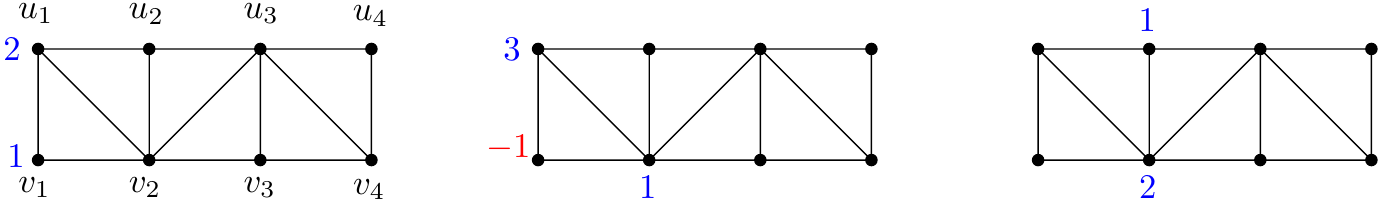}
    \caption{Three equivalent divisors on the \(2\times 4\) slashed ladder graph: \(2(u_1)+(v_1)\), \(3(u_1)-(v_1)+(v_2)\), and \((u_2)+2(v_2)\)}
    \label{figure:no_multiplicity_free_chips1}
\end{figure}

Performing the same collection of chip-firing moves in a different order does not change the resulting divisor.  Thus we can think of simultaneously chip-firing a subset \(U\subset V(G)\); the net effect is that one chip moves along every edge connecting \(U\) to \(U^C\), since any two vertices both in \(U\) being fired cancel out with respect to each other.  If chip-firing \(U\) does not introduce any new debt on \(G\), then we refer to \(U\) as a \emph{legal firing move}.

Given a divisor \(D\), a natural question is:  does there exist a divisor \(D'\) with \(D\sim D'\) and \(D\geq 0\)?  Or in the language of chip placements, can we perform chip-firing moves to eliminate all debt in \(D\)?  This question, sometimes called the Dollar Game, can be answered using \emph{\(q\)-reduced divisors} and \emph{Dhar's burning algorithm}.

Let \(q\in V(G)\).  We say that a divisor \(D\) is \emph{\(q\)-reduced} if the following two conditions are satisfied:
\begin{itemize}
    \item[(i)] \(D(v)\geq 0\) for all \(v\in V(G)-q\); and
    \item[(ii)] there does not exist a subset \(U\subset V(G)-q\) that is a legal firing move.
\end{itemize}
For each \(q\), every divisor \(D\) is equivalent to a unique \(q\)-reduced divisor, denoted \(D_q\) \cite[Proposition 3.1]{bn}; and \(D\) is equivalent to an effective divisor if and only if \(D_q\) is effective (note that it suffices to check that this holds for a single \(q\)).  

Thus being able to find \(q\)-reduced divisors is incredibly important.  Achieving condition (i) is always feasible; for instance, chip-firing \(q\) a large number of times will introduce enough chips into the rest of the graph to allow for the elimination of all debt away from \(q\).  From there, we use Dhar's burning algorithm \cite{dhar} to find subsets of \(V(G)-q\) that can perform chip-firing moves.  This algorithm works by starting a ``fire'' at \(q\), and letting the fire propagate through the graph according to the following rules:
\begin{itemize}
    \item If an edge is incident to a burning vertex, then that edge burns.
    \item  If a vertex is incident to more than \(D(v)\) burning edges, then that vertex burns.
\end{itemize}
If the entire graph burns, then \(D\) is \(q\)-reduced.  If there is an unburned set \(U\subset V(G)-q\) of vertices, then \(U\) is a legal firing move.  Chip-fire \(U\), and run the burning process on the new divisor.  In each iteration, either the whole graph burns or there is a new subset of vertices to fire.  Eventually the process terminates with the whole graph burning, at which point we have found our unique \(q\)-reduced divisor equivalent to \(D\).

To generalize the question of whether or not \(D\) is equivalent to an effective divisor, we introduce the notion of \emph{rank}.  Roughly speaking, the rank of a divisor indicates how much added debt the divisor can eliminate, regardless of where that debt is placed.  More formally, we define \(r(D)=-1\) if \(D\) is not equivalent to an effective divisor (meaning that \(D\) cannot even eliminate its own debt); and otherwise \(r(D)=r\) where \(r\) is the maximum nonnegative integer such that for any divisor \(E\geq 0\) of degree \(r\), we have that \(D-E\) is equivalent to an effective divisor. The following result, called the Riemann-Roch Theorem for graphs, is one of the most famous results regarding the ranks of divisors.  It is phrased in terms of the \emph{canonical divisor} \(K=\sum_{v\in V(G)}(\textrm{val}(v)-2)(v)\), and in terms of the graph's first Betti number \(g=|E(G)|-|V(G)|+1\).

\begin{theorem}[\cite{bn}] \label{theorem:riemann_roch} If \(D\) is a divisor on a graph \(G\), then
\[r(D)-r(K-D)=\deg(D)+1-g.\]
\end{theorem}

The \emph{divisorial gonality} (or simply gonality) of a graph \(G\), denoted \(\gon(G)\), is the minimum degree of a divisor of positive rank.  Informally, it is the smallest number of chips we can place on a graph such that no matter where \(-1\) debt is placed, one can eliminate all debt with chip-firing moves.  If a divisor \(D\) has positive rank and \(deg(D)=\gon(G)\), we say that \(D\) \emph{achieves} gonality.

Proving that the gonality of a graph is equal to \(k\) is quite involved.  First, one needs to prove that there exists a divisor of degree \(k\) and positive rank; and more challengingly, one needs to show that every effective divisor of degree \(k-1\) has rank \(0\).  
%
The following lemma will be useful for us in the latter part of such arguments.

\begin{lemma}\label{lemma:dhars_q_v}
Let \(q,v\in V(G)\) be distinct vertices, and let \(D\) be a \(q\)-reduced effective divisor such that \(D(v)=0\).  Run one iteration of Dhar's algorithm on \(D\) from the vertex \(v\).  If the vertex \(q\) burns, then \(r(D)=0\).
\end{lemma}

\begin{proof}
 Suppose for the sake of contradiction the algorithm does not burn the whole graph, although it does burn \(q\). Then  the unburned set of vertices \(U\subset V(G)-{v}\) gives  a legal chip-firing move. However, since \(q\) is burned we have that \(U\subset V(G)-{q}\) is a legal firing move, a contradiction to \(D\) being \(q\)-reduced.  Thus Dhar's algorithm burns the whole graph.  It follows that \(D\) is \(v\)-reduced, and since \(D(v)=0\), we have that \(r(D)=0\). 
\end{proof}

Some graphs are more susceptible to arguments using Dhar's burning algorithm than others.  Due to the number of edges along which fires can spread, complete graphs allow for detailed analysis using such methods, as illustrated in the following lemma.

\begin{lemma}\label{lemma:complete_graph_burning}
The gonality of the complete graph $K_n$ is equal to $n-1$, and is achieved only by placing $n-1$ chips on a single vertex, or by placing $1$ chip on \(n-1\) distinct vertices.  Moreover, any other effective divisor of degree at most $n-1$ will burn in a single iteration of Dhar's burning algorithm.
\end{lemma}

\begin{proof}
First, we argue that a divisor of degree $n-1$ in either of the given arrangements will be able to eliminate debt from the graph.  In the first case, where $n-1$ chips are placed on a single vertex, this vertex can be fired to send $1$ chip to every other vertex of the graph, thus eliminating debt from wherever it was placed.  In the second case, we can do the opposite: fire every vertex of the graph except for the one vertex $v$ without a chip, eliminating debt from $v$.  

Let \(D\) be another effective divisor of  degree \(n-1\), let \(q\) be any vertex with \(D(q)=0\), and suppose for the sake of contradiction that the whole graph does not burn on the first iteration of Dhar's burning algorithm on \(D-(q)\). Then there must be \(u\) unburned vertices where $1 \leq u \leq n-1$.  Since every vertex is adjacent to every other vertex, this means that each of the $u$ unburned vertices will be incident to $n-u$ burning edges, so each unburned vertex must have $n-u$ chips (otherwise they would burn).  So, we have $\deg(D)\geq u(n-u)$. As a function of $u$, this expression is concave down, and therefore will achieve its minimum on the interval at one of the boundary points, in this case $u=1$ or $u=n-1$.  If $u = 1$ or \(u=n-1\), we have $u(n-u) = n-1$; however, these values of $u$ correspond exactly to the two arrangements discussed above!  If $1$ vertex doesn't burn, it has $n-1$ chips on it, and thus all of the chips in the chip configuration.  And if $n-1$ vertices don't burn, then each vertex must have $1$ chip on them.  Since \(D\) was not one of these placements, we must have \(2\leq u\leq n-2\).  Since $u(n-u)$ is concave down, any value of \(u\) in this smaller range will result in $n(n-u)>n-1$, contradicting \(\deg(D)=n-1\) since \(\deg(D)\geq n(n-u)\). This completes the proof.
\end{proof}

We say that an effective divisor $D$ is \emph{multiplicity-free} if $D(v)\leq 1$ for every vertex $v$.  In other words, $D$ is multiplicity-free if it places at most $1$ chip on each vertex. The \emph{multiplicity-free gonality} \(\mfgon(G)\) of a graph \(G\) is the minimum degree of a multiplicity-free divisor of positive rank.

We immediately have that \(\gon(G)\leq \mfgon(G)\).  Not every graph has \(\gon(G)=\mfgon(G)\); for instance, a graph $G$ with $V(G)=\{v_1,v_2,v_3\}$ and edge multiset $\{v_1v_2,v_1v_2,v_2v_3,v_2v_3\}$ has gonality $2$, but every rank $1$ divisor of degree \(2\) is of the form $2(v_i)$ for some $i$.  It turns out there are also simple graphs with a gap between the two versions of gonality, as shown in the following examples.

\begin{example}\label{example:first_one}
Consider the graph $G$ on $8$ vertices illustrated in Figure \ref{figure:no_multiplicity_free}. We refer to this as the \(2\times 4\) \emph{slashed ladder graph}. If we construct a divisor $D$ by placing $3$ (or fewer) chips on distinct vertices, there exists some $i$ such that $u_i$ and $v_i$ both lack chips.  We claim that running Dhar's burning algorithm on $D-(u_i)$ then burns the whole graph.  Certainly $v_i$ also burns, and then any neighboring $u_j$, $v_j$ pair will burn as well:  even if each has a chip, one has $2$ incident burning edges and burns, and then the other has $2$ incident burning edges and burns.  This propagates until the whole graph burns, implying that $r(D)=0$.  However, there does exists a divisor of positive rank and degree $3$, namely $D=2(u_1)+(v_1)$; this divisor appears in Figure \ref{figure:no_multiplicity_free_chips1}.  This is equivalent to the divisors $(u_2)+2(v_2)$, $2(u_3)+(v_3)$, and $(u_4)+2(v_4)$; since together they cover all vertices of $G$, we have $r(D)>0$.  Some quick case-checking verifies that $\gon(G)>2$, so $G$ is a graph of gonality $3$ such that no multiplicity-free divisor achieves gonality.  Indeed,  generalizing to the \(2\times m\) slashed ladder graph, we can find examples of graphs with gonality \(3\) and multiplicity-free gonality at least (and in fact equal to) \(m\).

\end{example}

There also exist regular simple graphs with a gap between gonality and multiplicity-free gonality.

\begin{example}\label{example:antiprism}

\begin{figure}[hbt]
   		 \centering
\includegraphics[scale=0.8]{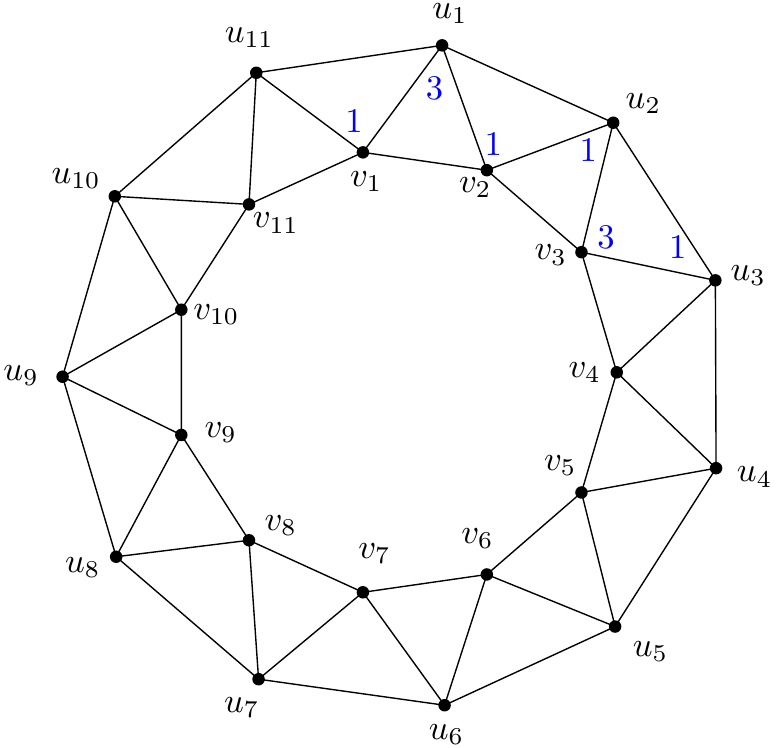}
	\caption{The antiprism on $11$ vertices, and a degree $10$ divisor achieving gonality}
	\label{figure:antiprism}
\end{figure}

The \emph{antiprisms} are $4$-regular graphs that give us a gap between gonality and multiplicity-free gonality. 
Consider the antiprism $\mathfrak{A}_{11}$ on $2\cdot 11=22$ vertices pictured in Figure \ref{figure:antiprism}.  This graph has vertices $u_1,\cdots,u_{11}$ and $v_1,\cdots,v_{11}$ arranged in two \(11\)-cycles, with $u_i$ attached to $v_i$ and $v_{i+1}$, working modulo $11$.  The divisor $3(u_1)+(u_2)+(u_3)+(v_1)+(v_2)+3(v_3)$ pictured has positive rank, as can be checked using Dhar's burning algorithm, implying that \(\gon(\mathfrak{A}_{11})\leq 10\).

We claim that no multiplicity-free divisor on $\mathfrak{A}_{11}$ has degree \(10\) or less, implying \(\gon(G)<\mfgon(G)\).  Let $D$ be an effective multiplicity-free divisor on  $\mathfrak{A}_{11}$ of degree $10$.  Since there are $11$ $\{u_i,v_i\}$ pairs, at least one such pair has no chips on it.  Choose such a pair, and run Dhar's burning algorithm on $D-(u_i)$.  Certainly $v_i$ will burn as well.  Letting $j=i\pm 1$, we claim that $u_j$ and $v_j$ now burn as well.  Indeed, one of them has two burning edges coming from the pair $\{u_i,v_i\}$, so it will burn; and then the other has one burning edge from $\{u_i,v_i\}$ and one from the other element of $\{u_j,v_j\}$.  Thus the fire spreads through the whole graph, implying that $r(D)=0$.  Thus $\mathfrak{A}_{11}$ is a $4$-regular graph such that no multiplicity-free divisor on it achieves gonality.

For a $5$-regular graph, we can add more edges to the antiprism.  In addition to connecting $u_i$ to $v_i$ and $v_{i+1}$, connect $u_i$ to $v_{i-1}$ (again working cyclically).  Building such a graph $G$ on $2\cdot 9=18$ vertices, we see that $\gon(G)\leq 8$, since the divisor illustrated in Figure \ref{figure:5-regular} has positive rank.  However, an argument identical to that for the antiprism shows that any multiplicity-free divisor of degree $8$ or less has rank $0$.  Thus this is a $5$-regular graph with multiplicity-free gonality strictly larger than gonality.  We will see in Proposition \ref{proposition:regular} that for any $r\geq 4$, there exist simple (and non-simple) $r$-regular graphs with a gap between gonality and multiplicity-free gonality.

\begin{figure}[hbt]
   		 \centering
\includegraphics[scale=1.2]{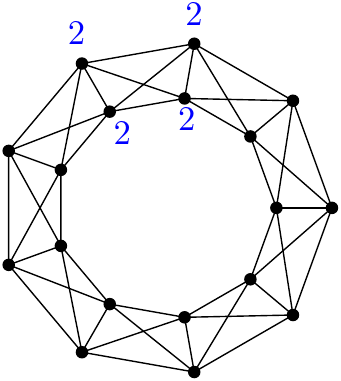}
	\caption{A $5$-regular graph with a positive rank divisor of degree $8$}
	\label{figure:5-regular}
\end{figure}

\end{example}

We now present several useful lemmas on multiplicity-free gonality.

\begin{lemma}\label{lemma:all_multiedges}
Let $G$ be a graph on \(n\) vertices such that every pair of vertices sharing at least $1$ edge share at least $2$ edges.  Then \(\textrm{mfgon}(G)=n\).
\end{lemma}

\begin{proof}

Let \(D\) be a multiplicity-free divisor on \(G\) with \(\deg(D)<n\).  At least one vertex $v\in V(G)$ has no chip from $D$.  Run Dhar's burning algorithm on $D-(v)$.  Anytime a vertex burns, all of its neighbors will burn, since they will have at least two incident burning edges, and each vertex has at most $1$ chip on it.  Since $G$ is connected, the whole graph will burn.  It follows that $r(D)=0$, implying that \(\mfgon(G)\geq n\).  Since placing \(1\) chip on every vertex gives a positive rank divisor, we have \(\mfgon(G)\leq n\), completing the proof.
\end{proof}

A set \(S\subset V(G)\) is called an \emph{independent set} if no two vertices of \(S\) share an edge.  The \emph{indpendence number} of \(G\), denoted \(\alpha(G)\), is the maximum possible size of an independent set.

\begin{lemma}\label{lemma:simple_independence}
If \(G\) is a simple graph, then \(\mfgon(G)\leq n-\alpha(G)\).
\end{lemma}

\begin{proof}
Let \(S\) be an independent set with \(|S|=\alpha(G)\), and consider the multiplicity-free divisor \(D\) with \(D(v)=0\) for \(v\in S\) and \(D(v)=1\) for \(v\notin S\).  As proved in \cite[Theorem 3.1]{random_graphs}, this divisor has positive rank; to see this, note that for \(v\in S\), chip-firing the set \(\{v\}^C\) moves chips onto \(v\) without introducing new debt, since each neighbor of \(v\) has $1$ chip and is connected to \(v\) by exactly one edge.  Since \(D\) is multiplicity-free, we have \(\mfgon(G)\leq\deg(D)= n-\alpha(G)\).
\end{proof}

We close this section by remarking on a possible generalization of multiplicity-free gonality, leaving it as a direction for future research.  Just as the gonality of a graph \(G\) can be defined as the minimum degree of a divisor with rank at least \(1\), for any positive integer \(r\) we can define the \(r^{th}\) gonality of a graph \(G\) to be the minimum degree of a divisor with rank at least \(r\).  This number is denoted \(\gon_r(G)\).

It is natural to try to define \(\mfgon_r(G)\) as the minimum degree of a multiplicity-free divisor with rank at least \(r\); however, this number is not always well-defined.  This is because there are only finitely many multiplicity-free divisors on a graph, and thus the maximum rank of such a divisor is bounded.  Thus we define  \(\mfgon_r(G)\) as the minimum degree of a multiplicity-free divisor with rank at least \(r\) if such a divisor exists, and \(\infty\) otherwise.  A natural question then becomes:

\begin{question}
Given a graph \(G\), for what values \(r\) do we have \(\mfgon_r(G)<\infty\)?
\end{question}

Note that the maximum rank of a multiplicity-free divisor is achieved by the divisor \(D_\mathbbm{1}\) that places \(1\) chip on every vertex.  Thus, \(\mfgon_r(G)<\infty\) if and only if \(r\leq r(D_\mathbbm{1})\).  So really, the question is: given a graph \(G\), what is \(r(D_\mathbbm{1})\)?

We answer this question in a few cases.  Recall that \(g=|E(G)|-|V(G)|+1\) is the first Betti number of the graph.

\begin{itemize}
    \item If \(G\) is a tree, then  \(r(D_\mathbbm{1})=\deg(D_\mathbbm{1})=|V(G)|\); and if \(G\) is a cycle, then  \(r(D_\mathbbm{1})=\deg(D_\mathbbm{1})-1=|V(G)|-1\).  This is because for a graph with \(g=0\) (i.e. a tree), the rank of  a divisor is equal to its degree; and for a graph with \(g=1\), the rank of an effective divisor is one less than its degree.

    \item  If \(G\) is a simple graph, we have \(r(D_\mathbbm{1})\geq 2\):  the only non-effective divisor of the form \(D_\mathbbm{1}-E\) where \(E\geq 0\) and \(\deg(E)=2\) is \(D_\mathbbm{1}-2(v)\) for some vertex \(v\); this divisor can be made effective by chip-firing all vertices but \(v\).  On the flip side, if \(G\) is a multigraph where every pair of incident vertices are connected by \(2\) or more edges, then \(r(D_\mathbbm{1})=1\); this is because running Dhar's burning algorithm on \(D_\mathbbm{1}-2(v)\) burns the whole graph (the proof is similar to that of Lemma \ref{lemma:all_multiedges}).  Thus any simple graph has \(\mfgon_2(G)\leq |V(G)|\), while any graph with all edges multiedges has \(\mfgon_2(G)=\infty\).

      \item  If \(G\) is a \(3\)-regular graph, then \(D_\mathbbm{1}\) is \(K\), the canonical divisor on \(G\) from the Riemann-Roch Theorem for graphs.  By that theorem, we know that \(r(K)=2g-2\).  Thus for any \(3\)-regular graph we have \(\mfgon_r(G)<\infty\) if and only if \(r\leq 2g-2\).
    
    \item  More generally, if \(G\) is a \(k\)-regular graph, then \(D_\mathbbm{1}\) is a divisor \(D\) such that \(kD=K\).  Note that \(2g-2=r(K)=r(kD)\geq kr(D)\), so we have \(r(D)\leq (2g-2)/k\).  
    
    However, for \(k>3\), \(r(D_\mathbbm{1})\) need not be determined by \(g\).  Consider the two \(4\)-regular graphs with \(g=7\) in Figure \ref{figure:two_genus_seven}.  As previously argued, the simple graph has \(r(D_\mathbbm{1})>2\), while the graph with all edges multiedges has \(r(D_\mathbbm{1})=1\).
    
\end{itemize}

\begin{figure}[hbt]
    \centering
    \includegraphics{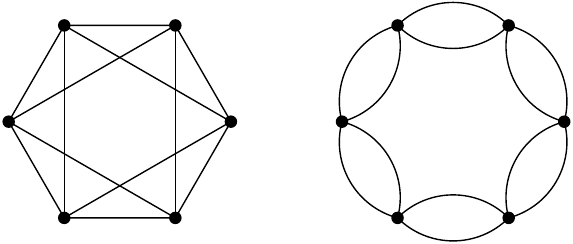}
    \caption{Two  \(4\)-regular graphs with \(g=7\) with different values for the rank of \(D_\mathbbm{1}\)}
    \label{figure:two_genus_seven}
\end{figure}
A more thorough study of this question would be an interesting direction for future work.

\section{A condition for equality}  
\label{section:equality}

In this section we provide a sufficient condition for a graph to have gonality equal to multiplicity-free gonality.  We start with the following lemma.

\begin{lemma} \label{eqlemma}
Let $G$ be a $k$-connected simple graph, and let $U \subset V(G)$ be a set of vertices with $2 \leq |U| \leq k-1$.  Then the outdegree of $U$ is at least $k+1$.
\end{lemma}

\begin{proof}
The outdegree of \(U\) can be computed as the total valence of the vertices in \(U\), minus twice the number \(e\)  of edges between vertices of \(U\) due to double-counting; thus the outdegree is
\[\textrm{outdeg}(U)=\sum_{v\in U}\textrm{val}(v)-2e.\]
As $G$ is a $k$-connected graph, any vertex in $G$ must be incident to at least $k$ edges, since deleting all neighbors of a vertex will disconnect the graph. Thus \(\sum_{v\in U}\textrm{val}(v)\geq k|U|\).  On the other hand, since \(G\) is simple, the total number of edges in \(U\) is at most \(\binom{|U|}{2}=\frac{|U|(|U|-1)}{2}\).  Thus we have
\[\textrm{outdeg}(U)\geq k|U|-|U|(|U|-1).\]
To show that  \(k|U|-|U|(|U|-1)\) is at least $k+1$, it is equivalent to show that
\[k \cdot \left(|U|-1\right) - |U| \cdot \left(|U|-1\right) \geq 1\]
since we can subtract $k$ from both sides.  Factoring transforms this expression into the following:
\[\left(k- |U|\right)\left(|U| - 1\right) \geq 1.\]
Since $k > |U|$ and $|U| > 1$, both $k-|U|$ and $|U| - 1$ are positive integers, so this inequality holds for all values of $k$ and $|U|$. This completes the proof.
\end{proof}

Now that we have this lemma, we can prove the following proposition:

\begin{proposition}\label{prop:k_gon_mfgon}
For any simple graph $G$, if $\gon(G) = \kappa(G)$, then\ \(\gon(G)=\mfgon(G)\).
\end{proposition}

\begin{proof} 

If $\kappa(G) = \gon(G) = 1$, then $G$ is a tree and we are done.  Henceforth we assume that $\kappa(G) \geq 2$.  Since $\gon(G) = \kappa(G) = k$, there exists an effective divisor \(D\) of degree \(k\) and positive rank.

Let \(S\) be the set of vertices on which \(D\) places chips, and let \(\ell=|S|\).  Suppose for the sake of contradiction that \(2\leq \ell-1\leq k-1\).  Let \(q\) be a vertex with no chip, and run Dhar's burning algorithm on \(D-(q)\).  Since $\kappa(G) = k$, removing the $\ell \leq k-1$ vertices of $S$ won't disconnect the graph, so every vertex not in \(S\) will burn.  Since \(r(D)>0\), the burning process stops before the entire set $S$ burns. Call the set of unburned vertices $U$. If $|U| = 1$, then the single unburned would need $k$ chips on it, contradicting the condition that $|S| \geq 2$, since $U \subseteq S$.  Thus we have $2 \leq |U| \leq k-1$.  But by Lemma \ref{eqlemma}, $U$ has outdegree at least $k+1$, so it requires at least $k+1$ chips for no vertices in it to burn, a contradiction to \(\deg(D)=k\).

Thus \(\ell=1\) or \(\ell=k\).  If \(\ell=k\), then \(D\) is multiplicity-free and we are done.   If \(\ell=1\), then \(D\) places all chips on a single vertex \(v\).  Since \(\kappa(G)\geq 2\), we know that \(G-v\) is connected.  Choose any vertex \(w\neq v\), and run Dhar's burning algorithm on \(D-(w)\).  The whole graph besides \(v\) will burn, since \(G-v\) is connected; but \(v\) will not burn since \(r(D)>0\).  Thus it is possible to chip-fire \(v\) by itself without introducing debt.  This means \(v\) is incident to at most \(k\) vertices; but since \(\kappa(G)=k\), it is also incident to at least \(k\) vertices, so it must be incident to exactly \(k\) vertices.  Thus chip-firing \(v\) turns \(D\) into a multiplicity-free divisor, with one chip on each of the \(k\) neighbors of \(v\). Either way, there exists a multiplicity-free divisor on \(G\) achieving gonality.

\end{proof}

There are many familiar families of graphs to which Proposition \ref{prop:k_gon_mfgon} applies, including:
\begin{itemize}
    \item trees, which have \(\kappa(G)=\gon(G)=1\) \cite[Lemma 1.1]{bn2} (indeed, that lemma proves that the trees are the only graphs of gonality \(1\));
    \item complete multipartite graphs \(K_{n_1,\ldots,n_\ell}\), which have \(\kappa(G)=\gon(G)=\left(\sum_{i=1}^\ell n_i\right)-\max_i n_i\) \cite[Example 4.3]{treewidth};
    \item  cycle graphs, which have \(\kappa(G)=\gon(G)=2\) \cite[\S 4.2]{db}. 
\end{itemize}

In fact, the last example falls into a more general class of graphs with gonality equal to multiplicity-free gonality:

\begin{proposition}\label{prop:hyp_simple}
If \(G\) is a simple graph, then \(\gon(G)=2\) if and only if \(\mfgon(G)=2\).
\end{proposition}  

\begin{proof}
If \(\mfgon(G)=2\), then \(G\) is not a tree, so \(1<\gon(G)\leq \mfgon(G)=2\).  It follows that \(\gon(G)=2\).

Now let \(G\) be simple with \(\gon(G)=2\), and let \(D=(u)+(v)\) be an effective divisor with \(\deg(D)=2\) and \(r(D)=1\).  If \(u\neq v\), then \(D\) is multiplicity-free, and we are done.  If not, choose any vertex \(q\) and run one iteration of Dhar's burning algorithm on \(D\) from \(q\).  Since \(r(D)\geq 1\), the whole graph will not burn, and so a subset \(U\) of chips will be made to fire.  Let \(D'\) be the divisor obtained by performing this subset-firing move.  At least one chip moved from \(u\), and even if both chips moved they could not have moved to the same vertex, since the graph is simple.  Thus \(D'\) is a multiplicity-free divisor of rank \(1\) and degree \(2\), implying that \(\mfgon(G)=2\).
\end{proof}

\section{Multiplicity-free gonality cannot be bounded by gonality}
\label{section:inequality}
We have that \(\gon(G)=1\) if and only if \(\mfgon(G)=1\) if and only if \(G\) is a tree; and for simple graphs that \(\gon(G)=2\) if and only if \(\mfgon(G)=2\) by Proposition \ref{prop:hyp_simple}.  In this section we will prove that these are the only cases in which gonality determines multiplicity-free gonality, or even in which some function of gonality can bound multiplicity-free gonality.

\begin{proposition}
If \(2\leq i\leq j\), then there exists a multigraph with \(\gon(G)=i\) and \(\mfgon(G)=j\).
\end{proposition}

\begin{proof}
Consider the multipath on \(j\) vertices, with \(i\) edges between each pair of adjacent vertices.  By \cite[\S 5]{gonality_sequences}, we have \(\gon(G)=\min(i,j)=i\); and by Lemma \ref{lemma:all_multiedges}, we have \(\mfgon(G)=j\).
\end{proof}

The corresponding result for simple graphs will take more work to prove.

\begin{theorem}\label{gon=mfgon}
For any integers $i, j$, with $3 \leq i \leq j$, there exists a simple graph $G$ with $\gon(G) = i$ and $\mfgon(G) = j$.
\end{theorem}

In order to prove this theorem, we introduce the \textit{complete slashed ladder graph}, $KL_{m,n}$, obtained by attaching a complete graph on \(n\) vertices to the end of a $2 \times m$ slashed ladder graph, overlapping on \(2\) vertices (throughout this section we assume \(m\geq 2\) and \(n\geq 3\)).  The complete slashed ladder \(KL_{6,5}\) is illustrated in Figure \ref{fig:comp-slashlad}.

\begin{figure}[hbtp]
    \centering 
    \includegraphics[scale = 0.8]{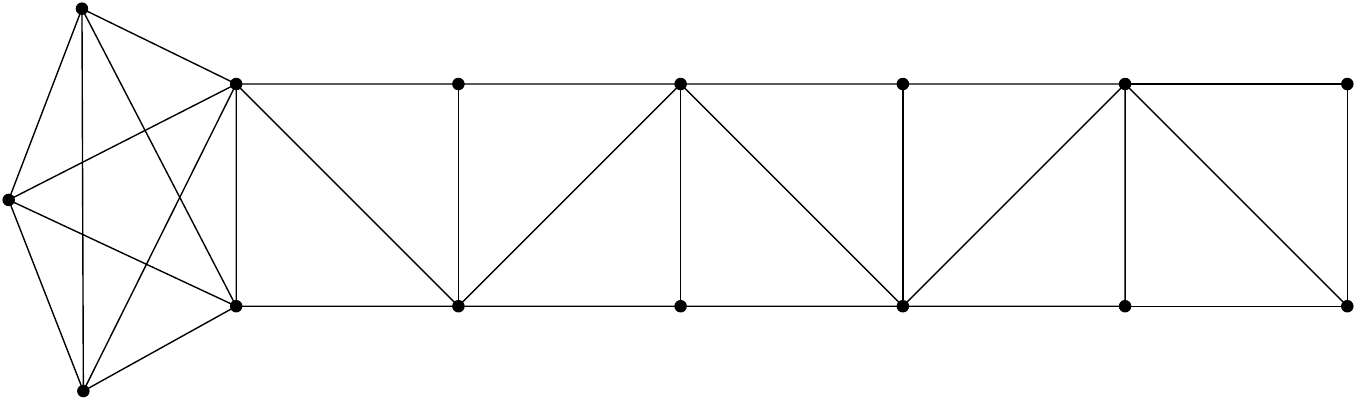}
    \caption{The complete slashed ladder $KL_{6,5}$.}
    \label{fig:comp-slashlad}
\end{figure}

\begin{lemma}
The complete slashed ladder graph $KL_{m,n}$ has gonality $n$.
\end{lemma}

\begin{proof}
First we present a positive rank divisor of degree $n$.  This divisor places $3$ chips on the first column of the slashed ladder graph, with $2$ chips on the vertex incident to the diagonal edge and $1$ chip on the other vertex, and adds $1$ chip to all but one of the remaining $n-2$ vertices of the complete portion of the graph.  This divisor is shown in Figure \ref{fig:comp-slashlad-div} for $KL_{6,5}$.

\begin{figure}[hbtp]
    \centering 
    \includegraphics[scale = 0.8]{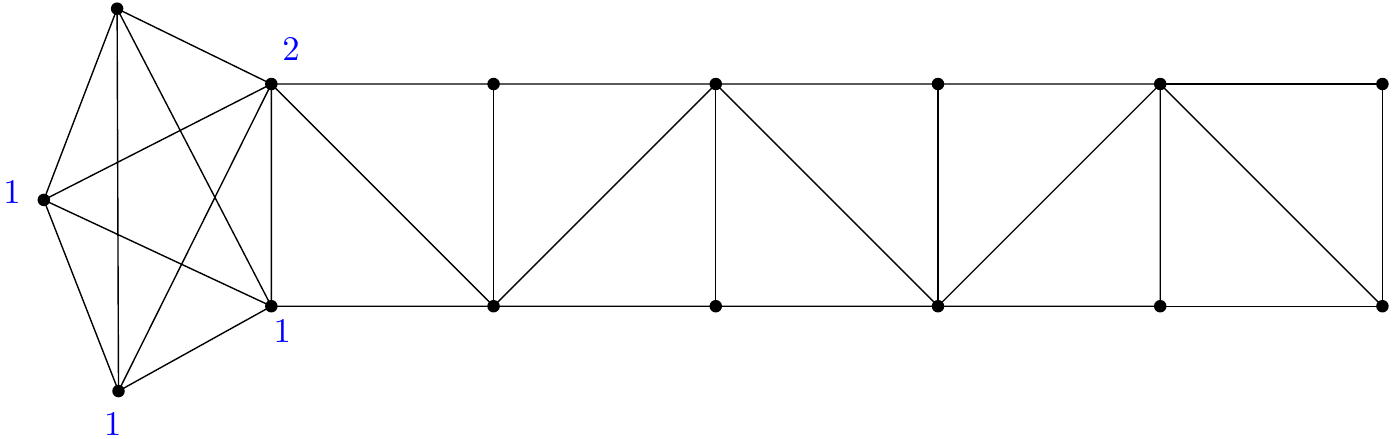}
    \caption{Divisor of degree $5$ for $KL_{6,5}$}
    \label{fig:comp-slashlad-div}
\end{figure}

This divisor has positive rank, because if debt is introduced on the slashed ladder portion of the graph, we can use the slashed ladder graph strategy to eliminate it; namely, chip-fire the entire complete graph portion of $KL_{m,n}$ to move the $3$ chips over to the second column of the slashed ladder graph, then chip-fire larger and larger subsets until the chips have eliminated the debt on the slashed ladder graph.  If debt is instead introduced on the last remaining vertex of the complete portion of the graph, chip-firing every single other vertex will eliminate debt from the graph.  So, \(\gon(KL_{m,n})\leq n\).

Suppose for the sake of contradiction that there is a positive rank effective divisor $D$ of degree $n-1$.  Of the two vertices in the overlap of the complete graph and the slashed ladder, choose \(q\) to be the one with higher valence. Without loss of generality we may assume that \(D\) is $q$-reduced,  so there is at least $1$ chip on $q$.  To reach a contradiction at this point, it suffices by Lemma \ref{lemma:dhars_q_v} to find a vertex \(v\in V(G)\) with \(D(v)=0\) such that running Dhar's burning algorithm on \(D-(v)\) burns the vertex \(q\).

Since there are only $n-1$ chips, there must be at least one vertex $v$ in the $K_n$ portion of the graph with no chips on it.  Run Dhar's burning algorithm on \(D-(v)\).  We know by Lemma \ref{lemma:complete_graph_burning} that for the $K_n$ subgraph (including \(q\)) not to burn with at most $n-1$ chips on it, it must have all $n-1$ chips, and either they must all be at $q$, or $1$ chip must be on each vertex of the $K_n$ except for $v$.  Based on this information regarding the structure of \(D\), we may now change our choice of \(v\). 

If $D$ has all $n-1$ chips on $q$, as in Figure \ref{fig:comp-slashlad-n-1-fail-1}, then any choice of \(v\) immediately burns the connected subgraph \(G-q\), and then $q$ as well since it is incident to more than \(n-1\) vertices.

\begin{figure}[hbtp]
    \centering 
    \includegraphics[scale = 0.8]{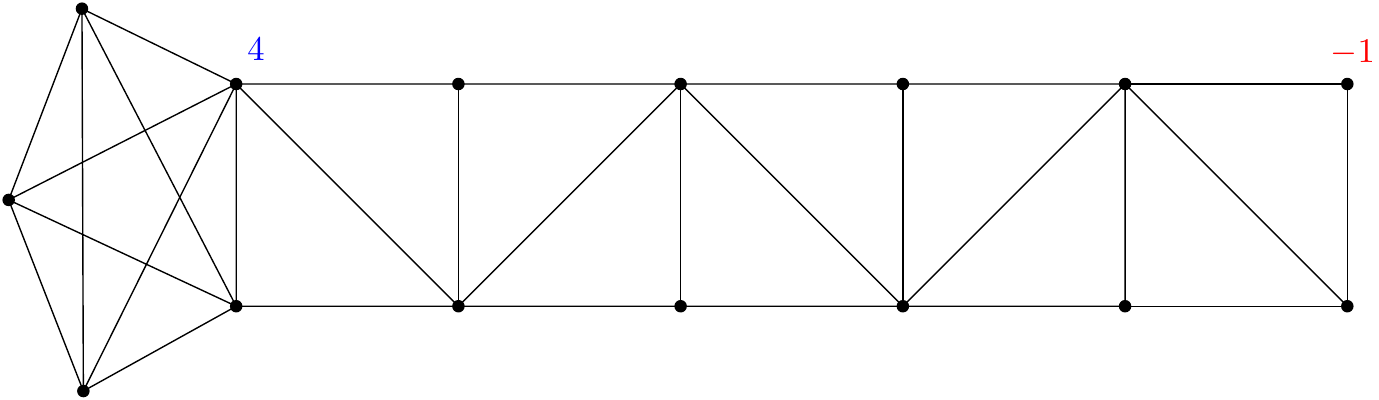}
    \caption{Degree $4$ divisor on $KL_{6,5}$ where all $4$ chips are on $q$, with $-1$ chips on a different vertex.}
    \label{fig:comp-slashlad-n-1-fail-1}
\end{figure}

If, instead, $D$ has $1$ chip on every vertex of $K_n$ except for one, as in Figure \ref{fig:comp-slashlad-n-1-fail-2}, then choose \(v\) to be in the complete slashed ladder portion of the graph.  Running Dhar's from \(v\) burns the whole slashed ladder away from \(K_n\), and then \(q\) burns as well since it has one chip and is incident to two burning edges, again giving us a contradiction.

\begin{figure}[hbtp]
    \centering 
    \includegraphics[scale = 0.8]{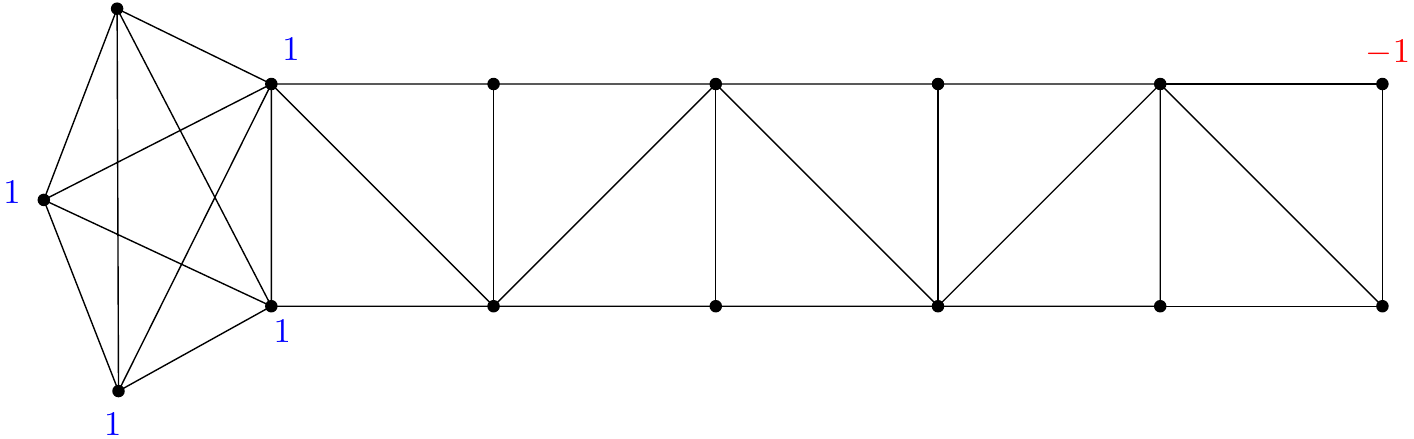}
    \caption{Degree $4$ divisor on $KL_{6,5}$ where each vertex has $1$ chip, with $-1$ chips on a vertex of the slashed ladder.}
    \label{fig:comp-slashlad-n-1-fail-2}
\end{figure}

Thus, it is impossible for a degree $n-1$ divisor to have positive rank. We conclude that the complete slashed ladder graph $KL_{m,n}$ has gonality $n$.

\end{proof}

\begin{lemma}
The complete slashed ladder graph $KL_{m,n}$ has multiplicity-free gonality $n + m - 2$.
\end{lemma}

\begin{proof}
One multiplicity-free divisor of degree $n + m - 2$ that has positive rank is the one that places a chip on each of the $m$ vertices along the diagonal of the slashed ladder portion of the graph, plus a chip on each of the $n-2$ vertices of the $K_n$ portion of the graph that aren't part of the slashed ladder.  Indeed, this is a placement of chips on the complement of an independent set, which always has positive rank for a simple graph  \cite[Theorem 3.1]{random_graphs}.  This divisor can be seen on $KL_{6,5}$ in Figure \ref{fig:comp-slashlad-mfdiv}.

\begin{figure}[hbtp]
    \centering 
    \includegraphics[scale = 0.8]{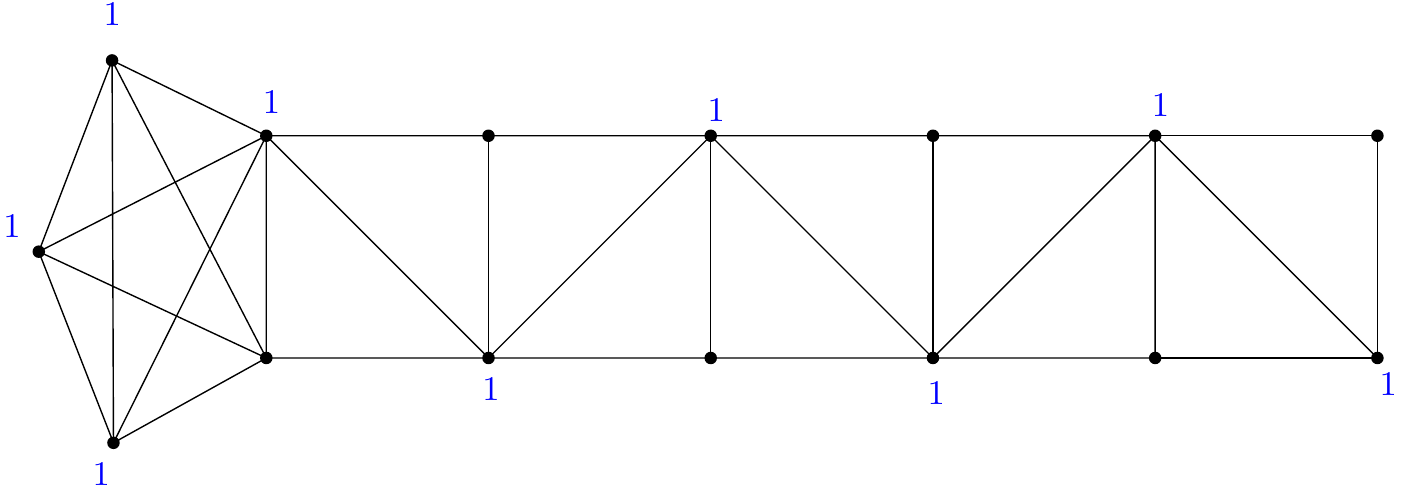}
    \caption{Degree $9$ multiplicity-free divisor on $KL_{6,5}$.}
    \label{fig:comp-slashlad-mfdiv}
\end{figure}

Now let \(D\) be an arbitrary multiplicity-free divisor. We claim that if a column of the slashed ladder graph portion has no chips on it from \(D\), then there exists a choice of vertex $v$ where $D - (v)$ burns. This can be seen as follows: if a column has no chips, we can start a fire on one of the vertices in that column.  Then, the other vertex in the column will burn, as well as all incident edges to each vertex in the column.  Because of the structure of the slashed ladder graph, there will be a vertex in each adjacent column that now has two burning edges incident to it, and at most $1$ chip on it, so it will burn, and then the other vertex in its column will burn as well for the same reason.  This process continues until the entire slashed ladder is burning.  Once the entire slashed ladder burns, the complete graph will already have $2$ burning vertices, and thus no vertex with at most $1$ chip on it will avoid burning.  So, for \(D\) to have positive rank, it must place at least \(m\) chips on the complete slashed ladder portion of the graph (at least one on each column).

Similarly, if the $K_n$ portion of the complete slashed ladder graph burns, then the entire graph will burn, because the $K_n$ portion contains a complete column of the slashed ladder graph within it.  Thus, if $K_n$ burns, the slashed ladder will burn as well, and therefore the whole graph will burn.  We know by Lemma \ref{lemma:complete_graph_burning} that there must be $n-1$ chips on $K_n$ to prevent it from burning, and the chips must be placed on all but $1$ vertex of $K_n$.  Thus there is a chip on at least \(n-1\) of the vertices of the complete graph, and at least one chip on each of the \(m\) columns of the slashed ladder graph. This means there are at least \(m+n-2\) chips in total, where the savings of \(1\) comes from the ability to have one of the columns of the slashed ladder taken care of by the complete graph's chips.  We conclude taht the multiplicity-free gonality of the complete slashed ladder graph is at least, and therefore exactly, $m + n - 2$.
\end{proof}

We are now ready to prove Theorem \ref{gon=mfgon}.

\begin{proof}[Proof of Theorem \ref{gon=mfgon}]
Let $n = i$ and $m = j - n + 2$.  Then, the complete slashed ladder graph $KL_{m,n}$ has gonality $i$ and multiplicity-free gonality $j$, for any desired values of $i$ and $j$.
\end{proof}

\section{Families of graphs} 
\label{section:families}
We close our paper by studying the mutiplicity-free gonality of graphs for certain graph families, or graphs with a particular structure.  In some cases we compute multiplicity-free gonalities where gonalities are unknown, and in other cases we can compare the two types of gonality.

\subsection{\(\ell\)-dimensional rook's graphs}

The \emph{Cartesian product} \(G\square H\) of two simple graphs \(G\) and \(H\) is the graph \(G\square H\) whose vertex set is \(V(G)\times V(H)\), where \((u_1,v_1)\) is adjacent to \((u_2,v_2)\) if and only if either \(u_1=u_2\) and \(v_1\) is adjacent to \(v_2\) in \(H\),  or \(v_1=v_2\) and \(u_1\) is adjacent to \(u_2\) in \(G\). An \(\ell\)-dimensional rook's graph is a Cartesian product of \(\ell\) complete graphs.  There are very few cases in which the gonality of rook's graph are known.  For \(n=2\), it was shown that \(\gon(K_m\square K_n)=\min\{m(n-1),n(m-1)\}\) for \(\min\{m,n\}\leq 5\) in \cite{cartesian_products}; a proof that this formula holds for all \(m\) and \(n\) will appear in a forthcoming paper by Noah Speeter.  It was also remarked that \(\gon(K_2\square K_3\square K_4)=12\) following the proof of \cite[Corollary 5.14]{small_2020_scramble}; beyond this, very little is known.  Restricting to multiplicity-free divisors, however, we can deliver a complete answer.

\begin{proposition}  Let $n_1\leq\cdots\leq n_\ell$, and consider $G=K_{n_1}\cart \cdots \cart K_{n_\ell}$. We have \[\mfgon(G)=(n_1-1)n_2\cdots n_\ell.\]
\end{proposition}

\begin{proof}
To see that $\mfgon(G)\leq(n_1-1)n_2\cdots n_\ell$, we consider the following multiplicity-free divisor of this degree:  we may view $G$ as $n_1$ copies of $K_{n_2}\cart \cdots \cart K_{n_\ell}$, connected according to $K_{n_1}$.  Choose one copy $H=K_{n_2}\cart \cdots \cart K_{n_\ell}$, and place a chip on every vertex in $V(G)\setminus V(H)$.  This divisor is multiplicity-free of degree $(n_1-1)n_2\cdots n_\ell$ and has positive rank, as we may fire all vertices outside of $H$ to eliminate debt wherever it is placed on $H$.  

We will now prove by induction on $\ell$ that if $D$ is a multiplicity-free divisor of degree $(n_1-1)n_2\cdots n_\ell-1$ on $ K_{n_1}\cart \cdots \cart K_{n_\ell}$, then there exists a vertex $v$ on the graph such that the entire graph burns on one iteration of Dhar's burning algorithm on $D-(v)$.  It will follow that no such multiplicity-free divisor has positive rank, giving us the desired lower bound on multiplicity-free gonality.

For the base case of $\ell=1$, our claim amounts to saying that if $K_n$ has an effective divisor $D$ of degree $(n-1)-1=n-2$, then there exists $v\in V(K_n)$ such that the whole graph burns on one iteration of Dhar's burning algorithm on $D-(v)$.  This is the content of \cite[Lemma 14]{cartesian_products}; the proof is very similar to that of Lemma \ref{lemma:complete_graph_burning}.

Now let $\ell\geq 2$, and assume our claim holds for $(\ell-1)$-fold products of complete graphs.  Let $D$ be an effective multiplicity-free divisor of degree $(n_1-1)n_2\cdots n_\ell-1$ on $G= K_{n_1}\cart \cdots \cart K_{n_\ell}$.  We can view $G$ as $n_\ell$ copies of $K_{n_1}\cart \cdots K_{n_{\ell-1}}$, with matching vertices connected according to $K_{n_\ell}$.  Refer to these copies as $H_1,\cdots,H_{n_\ell}$.  At least one $H_i$ has at most $(n_1-1)n_2\cdots n_{\ell-1}-1$ chips on it:  otherwise the degree of $D$ would be at least $(n_1-1)n_2\cdots n_\ell$.  By our inductive hypothesis, there exists a vertex $v\in V(H_i)$ such that all of $H_i$ burns under one iteration of Dhar's burning algorithm.  There exists some other index $j$ such that $H_j$ has at most $n_1n_2\cdots n_{\ell-1}-1$ chips: otherwise the degree of $D$ would be at least  $n_1n_2\cdots n_{\ell-1}(n_\ell-1)\geq (n_1-1)n_2\cdots n_{\ell-1}n_\ell$.  So at least one vertex in $H_j$, say $w$, does not have a chip.  Every vertex in $H_j$ is adjacent to a vertex in $H_i$, and so is incident to a burning edge.  This means the vertex $w$ burns.  From there every vertex in $H_j$ adjacent to $w$ will burn; then every vertex in $H_j$ adjacent to those vertices will burn; and so on.  Since $H_j$ is connected, every vertex in $H_j$ will burn.  Every other vertex $u$ in $G$ is incident to a vertex in $H_i$ and to a vertex in $H_j$, meaning that $u$ is incident to two burning edges, and $u$ will burn as well.  Thus the whole graph $G$ burns.

Since no multiplicity-free divisor of degree $(n_1-1)n_2\cdots n_\ell-1$ has positive rank, we have $\mfgon(G)\geq(n_1-1)n_2\cdots n_\ell $.  We conclude that  $\mfgon(G)=(n_1-1)n_2\cdots n_\ell $.
\end{proof}

\subsection{Wheel graphs}  The wheel graph \(W_n\) on \(n+1\) vertices consists of a cycle on \(n\) vertices together with a universal vertex.  To determine which wheel graphs have multiplicity-free gonality equal to gonality, we must first determine the gonality of \(W_n\).

\begin{theorem}\label{theorem:wheel_gon}
$\gon(W_n) = \lceil \sqrt{n}\rceil - 1 + \left\lceil \frac{n}{\lceil \sqrt{n}\rceil} \right \rceil$
\end{theorem}

\begin{proof}
Let $D$ be an effective divisor achieving gonality on $W_n$, and suppose for the sake of contradiction that $\deg(D)\leq\lceil \sqrt{n}\rceil - 2 + \left\lceil \frac{n}{\lceil \sqrt{n}\rceil} \right\rceil$.  Let $w$ be the universal vertex.  Without loss of generality, we will assume that $D$ is $w$-reduced.

Let $k=D(w)$; we know that $k\geq 1$, since $D$ is $w$-reduced and \(r(D)=1\).  We claim that there exist $k+1$ incident radial vertices $v_1,\ldots,v_{k+1}$ such that $D(v_i)=0$ for all $i$.  Suppose not.  This means that there are at least $\lceil\frac{n}{k+1}\rceil$ chips placed on the radial vertices; otherwise the prescribed gap would exist.  It follows that $k\leq \deg(D)-\left\lceil\frac{n}{k+1}\right\rceil\leq\lceil \sqrt{n}\rceil - 2 + \left\lceil \frac{n}{\lceil \sqrt{n}\rceil} \right\rceil-\left\lceil\frac{n}{k+1}\right\rceil$.  This can be rewritten as
\[(k+1)+\left\lceil\frac{n}{k+1}\right\rceil+1\leq \lceil \sqrt{n}\rceil+ \left\lceil \frac{n}{\lceil \sqrt{n}\rceil} \right\rceil\]
Note that $(k+1)+\lceil\frac{n}{k+1}\rceil+1=\left\lceil(k+1)+\frac{n}{k+1}+1\right\rceil$.  Consider the function $x+\frac{n}{x}+1$.  This function is concave up for all positive $x$, and is minimized at $x=\sqrt{n}$.  If $k$ can be any positive integer, it follows that $(k+1)+\frac{n}{k+1}+1$ is at its minimimum when $k+1$ is either $\floor{\sqrt{n}}$ or $\ceil{\sqrt{n}}$; the same is true for $\left\lceil(k+1)+\frac{n}{k+1}+1\right\rceil$. In fact, by Proposition \ref{rootfloorproof}, these two choices for $k+1$ give the same output.  It follows that
\[\lceil \sqrt{n}\rceil+ \left\lceil \frac{n}{\lceil \sqrt{n}\rceil} \right\rceil+1\leq(k+1)+\left\lceil\frac{n}{k+1}\right\rceil+1\leq \lceil \sqrt{n}\rceil+ \left\lceil \frac{n}{\lceil \sqrt{n}\rceil} \right\rceil,\]
a contradiction.  This lets us conclude that there exist $v_1,\ldots,v_{k+1}$ as claimed.

Run Dhar's burning algorithm on $D-(v_1)$.  The vertices $v_1,\ldots,v_{k+1}$ all burn, since they are connected and $D(v_i)=0$ for all $i$.  The vertex $w$ is now incident to $k+1$ burning edges, and so will burn as well since $D(w)=k$.  If the graph does not burn, then some subset of $V(W_n)$ not including $w$ can fire, contradicting the fact that $D$ is $w$-reduced.  Thus the whole graph will burn and $r(D)<1$, a contradiction.  We conclude that the gonality of $W_n$ is at least   $\lceil \sqrt{n}\rceil - 1 + \left\lceil \frac{n}{\lceil \sqrt{n}\rceil} \right\rceil$.

To see that the gonality of $W_n$ is at most $\lceil \sqrt{n}\rceil - 1 + \left\lceil \frac{n}{\lceil \sqrt{n}\rceil} \right\rceil$, place $\lceil \sqrt{n}\rceil - 1$ chips on $w$, and place $\left\lceil \frac{n}{\lceil \sqrt{n}\rceil} \right\rceil$ chips around the radial vertices so that no two chipped vertices in a row are more than $\lceil \sqrt{n}\rceil$ apart on the outer cycle. This is illustrated for \(n=12\) on the wheel with \(13\) vertices in Figure \ref{figure:wheel_general}. The only vertices without chips are then radial vertices in clusters of length at most $\lceil \sqrt{n}\rceil -1$.  If $-1$ is placed on such a vertex, then firing all vertices not in its cluster eliminates all debt from the graph.  Thus, there exists a positive rank divisor of degree  $\lceil \sqrt{n}\rceil - 1 + \left\lceil \frac{n}{\lceil \sqrt{n}\rceil} \right\rceil$.  

\begin{figure}[hbt]
    \centering
    \includegraphics{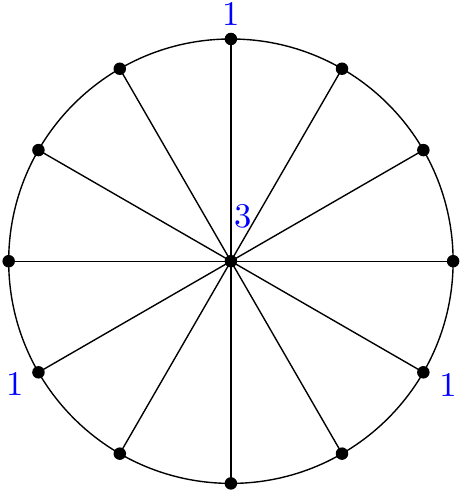}
    \caption{A winning placement of chips on the wheel with \(13\) vertices}
    \label{figure:wheel_general}
\end{figure}
\end{proof}

Now we determine the multiplicity-free gonality of \(W_n\).  We will use the following lemma.

\begin{lemma}\label{lemma:universal_vertex}
Let \(G\) be a simple graph with a universal vertex \(v\) such that \(G-v\) is connected. Then \(\textrm{mfgon}(G)=n-\alpha(G)\).
\end{lemma}

\begin{proof}
We have \(\textrm{mfgon}(G)\leq n-\alpha(G)\) by Lemma \ref{lemma:simple_independence}.  Now let \(D\) be a multiplicity-free divisor of degree less than \(n-\alpha(G)\).  It follows that \(D\) that has two adjacent vertices \(u\) and \(w\) that do not have chips.  Place a \(-1\) on either of them, and run Dhar's burning algorithm, so that both \(u\) and \(w\) burn.  Either \(v\in \{u,w\}\), in which case both \(v\) and a vertex of \(G-v\) burn; or \(v\notin \{u,w\}\), in which case \(v\) has two burning edges and at most \(1\) chip and thus we still have both \(v\) and a vertex of \(G-v\) burn.  At this point every vertex in \(G-v\) has at least one burning edge coming from \(v\), meaning that one more burning edge is enough to make them burn; and since \(G-v\) is connected with at least one burning vertex, the fire spreads through the whole graph.  Thus \(r(D)=0\), and so there does not exist a multiplicity-free divisor of degree less than \(n-\alpha(G)\) of positive rank.  This gives us \(\mfgon(G)\geq n-\alpha(G)\), completing the proof.
\end{proof}

\begin{corollary}\label{corollary:wheel_mfgon}
\(\mfgon(W_n)=\lceil n/2\rceil+1\).
\end{corollary}

\begin{proof}
This follows from Lemma \ref{lemma:universal_vertex}, the fact that \(\alpha(W_n)=\lfloor n/2\rfloor\), and the identity \((n+1)-\lfloor n/2\rfloor=\lceil n/2\rceil+1\).
\end{proof}

\begin{theorem}
The wheel \(W_n\) has gonality equal to multiplicity-free gonality if and only if \(n\leq 8\).
\end{theorem}

We reserve the proof of this result for Proposition \ref{prop:wheel_equality} in Appendix \ref{section:wheel_lemmas}; it is simply a verification that these are the only values of \(n\) that make the formulas from Theorem \ref{theorem:wheel_gon} and Corollary \ref{corollary:wheel_mfgon} equal to one another.  The small wheel graphs that do have gonality achieved by a multiplicity-free divisor are pictured in Figure \ref{fig:mf_wheels}, along with such a divisor.

\begin{figure}[hbt]
    \centering
    \includegraphics{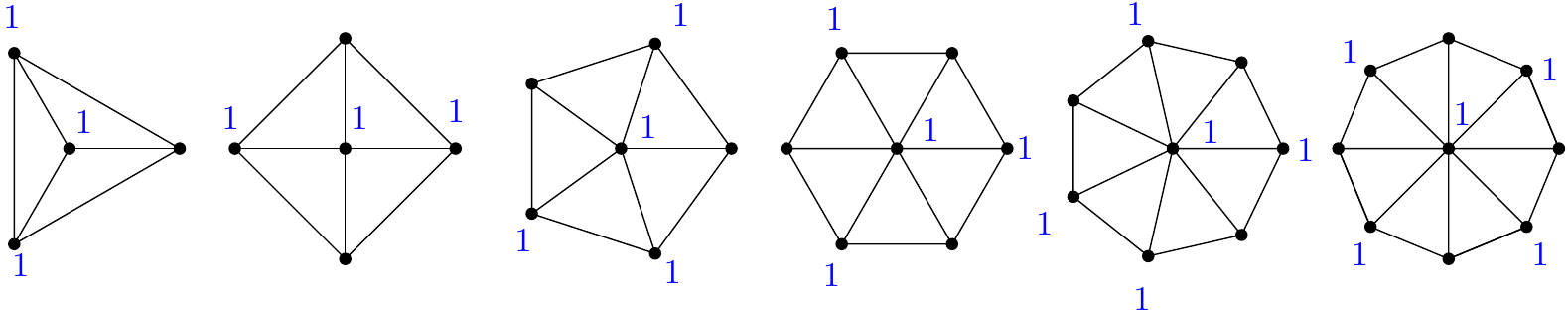}
    \caption{Wheel graphs with multiplicity-free divisors achieving gonality}
    \label{fig:mf_wheels}
\end{figure}

\subsection{Simple graphs with high minimum valence}

In \cite{small_2020_scramble}, it was shown that if a simple graph \(G\) on \(n\) vertices has minimum valence \(\delta(G)\geq \lfloor n/2\rfloor +1\), then
\[\gon(G)=n-\alpha(G).\]
Since \(\gon(G)\leq \mfgon(G)\leq n-\alpha(G)\), we immediately have the following result:

\begin{corollary}\label{corollary:gon_mfgon_n-alpha}
Let \(G\) be a simple graph on \(n\) vertices with minimum valence \(\delta(G)\geq \lfloor n/2\rfloor +1\).  Then \(\gon(G)=\mfgon(G)=n-\alpha(G)\).
\end{corollary}

The authors of \cite{small_2020_scramble} leveraged their result to prove that it is NP-hard to compute the gonality of a simple graph.  We can mirror their argument to do the same for multiplicity-free gonality.

\begin{theorem}\label{theorem:mfgon_nphard}
Computing \(\mfgon(G)\) is NP-hard.
\end{theorem}

\begin{proof}
Let \(H\) be any simple connected graph on \(m\) vertices, and let \(G\) be the \(m^{th}\) cone over \(H\).  That is, \(G\) is obtained from \(H\) by iteratively adding \(m\) universal vertices to it. Letting \(n=2m\), we have that \(G\) is a graph on \(n\) vertices with \(\delta(G)\geq n/2 +1\), implying by Corollary \ref{corollary:gon_mfgon_n-alpha} that \(\mfgon(G)=n-\alpha(G)\).  Noting that \(\alpha(H)=\alpha(G)\), we have \(\mfgon(G)=n-\alpha(H)\).
 This means we may compute the independence number of any simple graph \(H\) by computing the multiplicity free gonality of a graph \(G\) that is polynomial in the size of \(H\).  Since independence number is NP-hard to compute, we conclude that \(\mfgon(G)\) is NP-hard to compute as well.
\end{proof}

\subsection{Regular graphs}  In this subsection we consider the question:  for which values of \(r\) do there exist \(r\)-regular graphs with a gap between gonality and multiplicity-free gonality?  We remark that the only connected \(1\)-regular graph is \(K_2\), which has gonality and multiplicity-free gonality equal to \(1\); and the only connected \(2\)-regular graphs are the cycle graphs \(C_n\), which have gonality and multiplicity-free gonality both equal to \(2\).  Once we have \(r\geq 4\), however, we can find graphs exhibiting a gap between the two gonalities, both among simple graphs and multigraphs.

\begin{proposition}\label{proposition:regular}  For any $r\geq 4$, there exist simple and non-simple $r$-regular graphs \(G\) with \(\gon(G)<\mfgon(G)\).
\end{proposition}

\begin{proof}
The non-simple graph is easier, so we construct that one first.  We illustrate our constructed graphs for $r=4$ and $r=5$ in Figure \ref{figure:regular_multigraphs}.  If $r$ is even, construct a multigraph $G$ on $r+1$ vertices with the cycle $C_{r+1}$ as the underlying simple graph, where each adjacent pair of vertices is connected by $r/2$ edges.  Then for any $v\in V(G)$, the divisor $r(v)$ has nonnegative rank:  it is equivalent to $\frac{r}{2}(v')+\frac{r}{2}(v'')$ for any pair of vertices $v',v''$ equidistant from $v$ on the cycle.  Thus $\gon(G)\leq r<r+1=|V(G)|$.  Since there are $r/2\geq 2$ edges between each pair of neighboring vertices, we may apply Lemma \ref{lemma:all_multiedges} to conclude that $G$ has gonality strictly smaller than multiplicity-free gonality.

If $r$ is odd, write $r=2s+1$.  Construct $G$ on $2s(s+1)+2$ vertices with the cycle $C_{2s(s+1)+2}$ as the underlying graph, where the number of edges between two vertices alternates between $s$ and $s+1$ as we move around the cycle.  Choose two vertices $v_1$ and $v_2$ connected by $s$ edges, and consider the divisor $D=s(s+1)(v_1)+s(s+1)(v_2)$.  By chip-firing the set $\{v_1,v_2\}$ a total of $s$ times, we have that $D$ is equivalent to $s(s+1)(w_1)+s(s+1)(w_2)$, where $w_i$ is the  neighbor of $v_i$ not in $\{v_1,v_2\}$.  Then, by chip-firing the set $\{v_1,v_2,w_1,w_2\}$ a total of $s+1$ times, we have that $D$ is equivalent to $s(s+1)(u_1)+s(s+1)(u_2)$, where $u_i$ is the neighbor of $w_i$ not in $\{v_1,v_2\}$.  Continuing in this way, we may move our chips around the whole graph, so $r(D)>0$.  This means that $\gon(G)\leq \deg(D)=2s(s+1)<2s(s+1)+2=|V(G)|$.   Since \(r\geq 5\), there are at least $s\geq 2$ edges between each pair of neighboring vertices, we may apply Lemma \ref{lemma:all_multiedges} to conclude that $G$ has \(\gon(G)<\mfgon(G)\).

\begin{figure}[hbt]
   		 \centering
\includegraphics[scale=0.7]{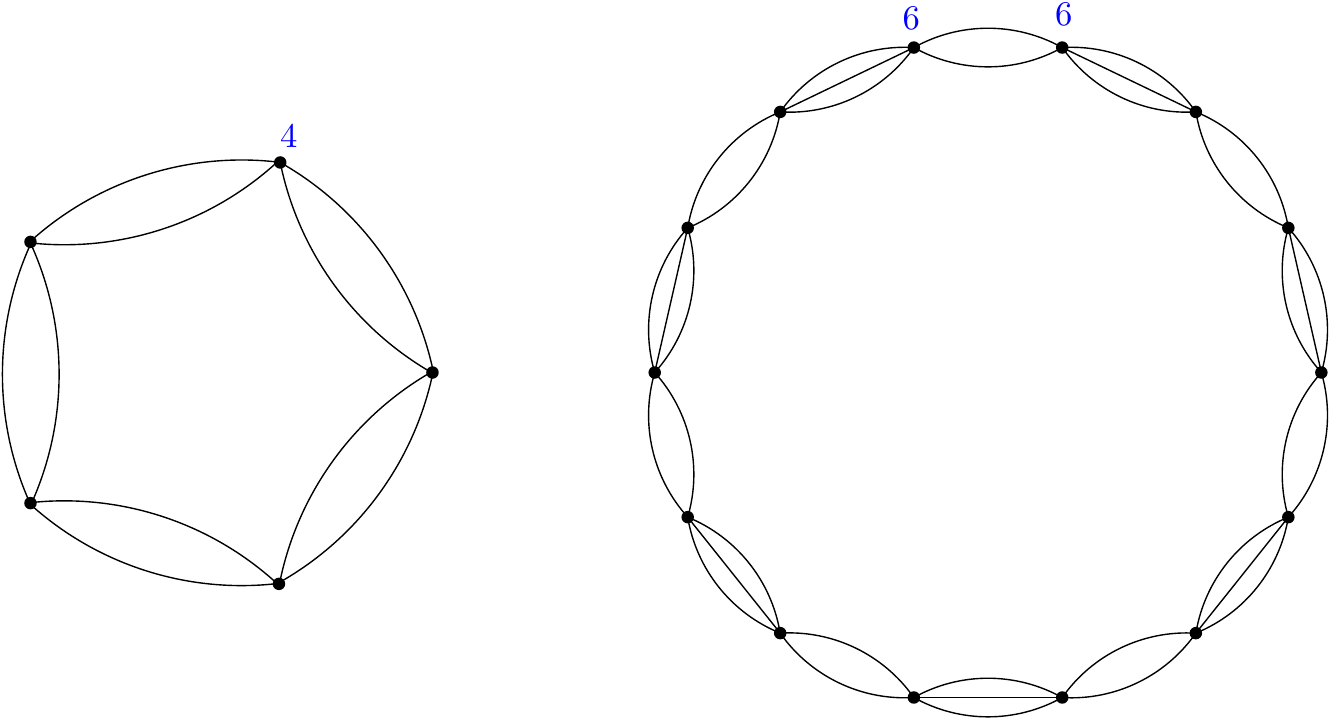}
	\caption{The $4$-regular and $5$-regular non-simple graphs constructed in the proof, along with positive rank divisors}
	\label{figure:regular_multigraphs}
\end{figure}

We now construct a simple graph with our desired properties. The graph we will end up constructing for $r=7$ appears in Figure \ref{figure:7-regular}; there are $9$ copies of $K_4$, connected as pictured.

\begin{figure}[hbt]
   		 \centering
\includegraphics[scale=0.7]{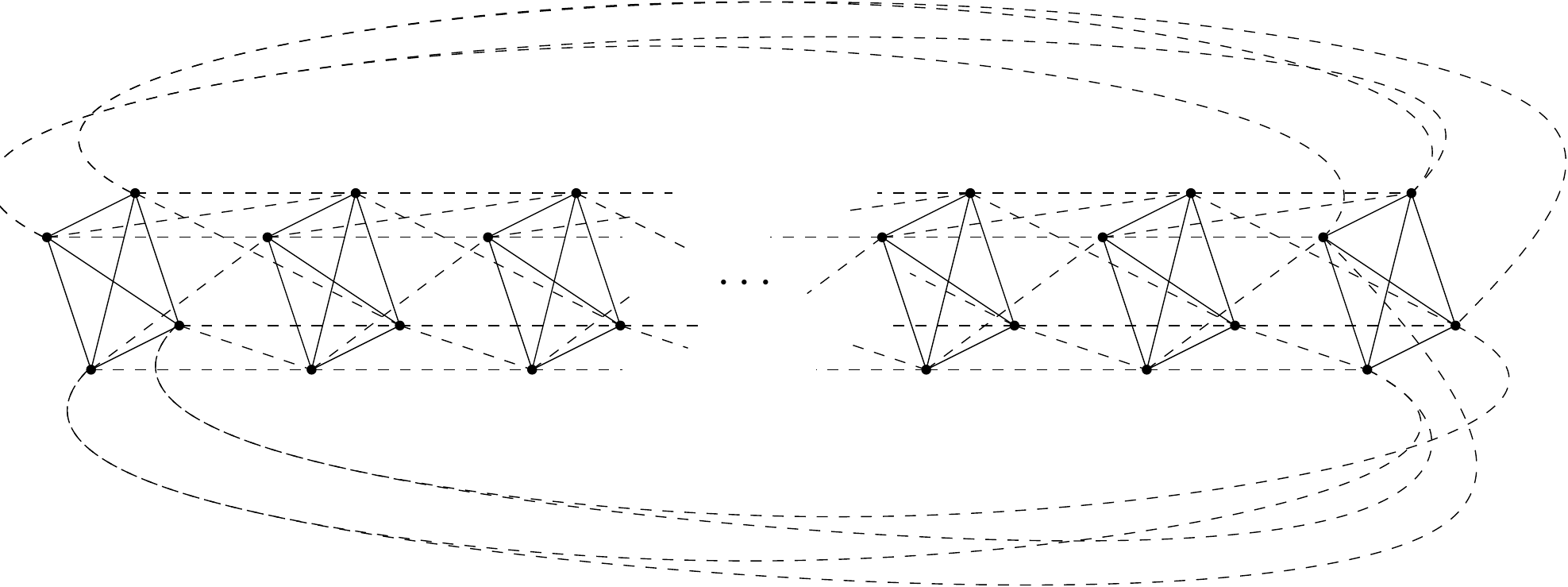}
	\caption{A simple $7$-regular graph with $\gon(G) < \mfgon(G)$; there are $9$ copies of $K_4$}
	\label{figure:7-regular}
\end{figure}

If $r=4$ or $r=5$, we can use one of the graphs from Example \ref{example:antiprism}.  Now let $r\geq 6$.  
Choose $N$ such that $\frac{4r-3}{r-4}<N$, and construct $N$ copies of $K_{r-3}$, referring to them as $K_{r-3}^{(1)}$,$\cdots$,$K_{r-3}^{(N)}$.  Label the vertices of $K_{r-3}^{(i)}$ as $v_1^{(i)},\cdots,v_{r-3}^{(i)}$.  Let $\overline{n}$ denote $n\mod N$. For pairs $i_1$ and $i_2$ with $i_2=\overline{i_1+1}$, connect $v_j^{(i_1)}$ to both $v_j^{(i_2)}$ and $v_{\overline{j+1}}^{(i_2)}$ for all $j$. (Without the $v_j^{(i_1)}v_{\overline{j+1}}^{(i_2)}$ edges, this graph would simply be the Cartesian product $K_{r-3}\boxempty C_N$.)

Note that every vertex $v$ in $G$ has valence $(r-4)+4=r$, where $r-4$ edges come from the $K_{r-3}$ containing $v$ and the other $4$ edges connect $v$ to other copies of $K_{r-3}$.  Also note that $\gon(G)\leq 4(r-3)$:  placing $4$ chips on every vertex of one copy of $K_{r-3}$ yields a positive rank divisor, as we may spread these chips around the graph by iteratively firing copies of $K_{r-3}$.  Suppose $G$ has  a multiplicity-free divisor $D$ achieving gonality.  Then $\deg(D)\leq 4(r-3)$.  Since we chose $N$ such that $\frac{4r-3}{r-4}<N$, we have $\frac{4r-3}{N}<r-4$, so $\left\lfloor{\frac{4r-3}{N}}\right\rfloor\leq r-5$, and thus by the Pigeonhole Principle at least one copy of $K_{r-3}$ has at most $r-5$ chips on it.  Choose a vertex $q$ of this $K_{r-3}$ with no chips on it, and run Dhar's burning algorithm on $D-(q)$.  By Lemma \ref{lemma:complete_graph_burning}, the copy of  $K_{r-3}$ containing $q$ burns.  Then, the two neighboring copies of $K_{r-3}$ burn as well:  each of their vertices has two incident burning edges, but at most one chip.  Their neighboring $K_{r-3}$'s burn as well, until the whole graph burns, contradicting $r(D)>0$.  Thus $G$ is an $r$-regular simple graph with gonality strictly smaller than multiplicity-free gonality.

\end{proof}

We pose the following as an open question.

\begin{question}
Does there exist a \(3\)-regular graph with \(\gon(G)<\mfgon(G)\)?
\end{question}

If $G$ is $3$-regular, then every divisor $D$ on $G$ with positive (indeed, nonnegative) rank with degree at most \(g-1\) is equivalent to a divisor $D'$ that places at most $2$ chips on each vertex; this is a consequence of work in \cite{spencer}, stated explicitly in the discussion following \cite[Lemma 16]{cartesian_products}.  We might hope that by performing chip-firing moves, we could transform such a divisor into a multiplicity-free one.  The next example shows that this strategy will not work for all divisors on $3$-regular graphs.

\begin{example} 

\begin{figure}[hbt]
   		 \centering
\includegraphics[scale=0.8]{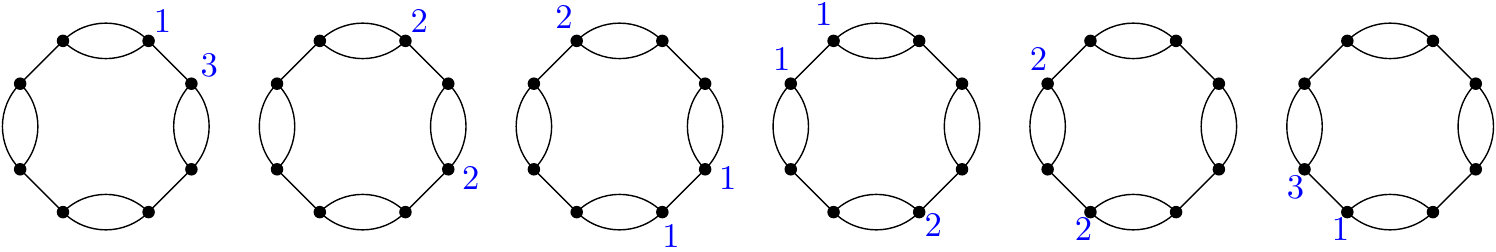}
	\caption{The effective representatives of a divisor class of positive rank, with no multiplicity-free representatives}
	\label{figure:trivalent_example}
\end{figure}

Consider the ``loop of loops'' $L_5$ of genus $5$.  This $3$-regular graph is pictured in Figure \ref{figure:trivalent_example}, with six equivalent divisors; in fact, pictured are all effective representatives of an equivalence class of divisors.  Letting $D$ refer to any one of these divisors, we have that $r(D)>0$, and $\deg(D)=4$.  It turns out that $L_5$ has gonality $4$, as can be shown with an exhaustive Dhar's burning algorithm argument.  Thus, even though $D$ achieves gonality, it is not equivalent to any multiplicity-free divisor.

\begin{figure}[hbt]
   		 \centering
\includegraphics[scale=1]{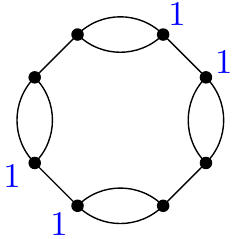}
	\caption{A multiplicity-free divisor on $L_5$ achieving gonality}
	\label{figure:trivalent_property_m}
\end{figure}

However, $L_5$ does have \(\gon(L_5)=\mfgon(L_5\)):  a multiplicity-free divisor achieving gonality is pictured in Figure \ref{figure:trivalent_property_m}.  Nonetheless, this example illustrates that if we wish to prove that any $3$-regular graph has \(\gon(G)=\mfgon(G)\), it will not work to start with an arbitrary divisor achieving gonality and then perform chip-firing moves until it is multiplicity-free.

\end{example}

\bibliographystyle{alpha}
\bibliography{bibliography}

\appendix

\section{Lemmas for the wheel graph}\label{section:wheel_lemmas}

\begin{proposition}\label{rootfloorproof}  For any positive integer \(n\), we have 
$\lceil \sqrt{n}\rceil - 1 + \left\lceil \frac{n}{\lceil \sqrt{n}\rceil} \right\rceil = \lfloor \sqrt{n}\rfloor - 1 + \left\lceil \frac{n}{\lfloor \sqrt{n}\rfloor} \right\rceil$.
\end{proposition}

\begin{proof}
The claim is immediately true for square numbers $n$, since everything being rounded is already an integer. Thus, assume $n$ is not a square number. Let $s = \lfloor \sqrt{n} \rfloor$. It suffices to show $\lceil \frac{n}{s}\rceil = \lceil \frac{n}{s+1}\rceil + 1$. Note that $s^2 <n$ and since $n$ is an integer, $(s+1)^2 -1 \geq n \Rightarrow s^2 +2s \geq n$.

Suppose $n \leq s^2 + s$. Then it follows that 

\begin{align*}
s < \frac{n}{s} \leq s+1 \\ & \Rightarrow \left \lceil \frac{n}{s}\right \rceil = s+1 \\ s-1 < \frac{n}{s+1} \leq s \\ & \Rightarrow \left \lceil \frac{n}{s+1}\right \rceil = s \\
\end{align*}
Thus, the claim is true.
\\\\Suppose $n > s^2 + s$. Then it follows that 

\begin{align*}
s+1 < \frac{n}{s} \leq s+2 \\ & \Rightarrow \left \lceil \frac{n}{s}\right \rceil = s+2 \\ s < \frac{n}{s+1} < s+1 \\ & \Rightarrow \left \lceil \frac{n}{s+1}\right \rceil = s+1 \\
\end{align*}
And the claim is also true.
\end{proof}

\begin{proposition}\label{prop:wheel_equality}  Let \(n\geq 3\). We have
$\lfloor \sqrt{n}\rfloor - 1 + \lceil \frac{n}{\lfloor \sqrt{n}\rfloor} \rceil = \lceil \frac{n}{2} \rceil+1 $ if and only if \(n\leq 8\).
\end{proposition}

\begin{proof}
We have
\begin{align*}
    \floor{\sqrt{n}}-1+\left\lceil{\frac{n}{\floor{\sqrt{n}}}}\right\rceil < & \sqrt{n}+\frac{n}{\floor{\sqrt{n}}}
    \\=& n\cdot \left(\frac{1}{\sqrt{n}}+\frac{1}{\floor{\sqrt{n}}}\right)
    \\\leq &n\cdot \frac{2}{\lfloor \sqrt n\rfloor}. 
\end{align*}
Note that for \(n\geq 16\), we have \(\frac{2}{\lfloor \sqrt n\rfloor}\leq \frac{2}{\lfloor \sqrt 16\rfloor}=\frac{1}{2}\), so we have
\[ \floor{\sqrt{n}}-1+\left\lceil{\frac{n}{\floor{\sqrt{n}}}}\right\rceil< \frac{n}{2}<\left\lceil\frac{n}{2}\right\rceil+1.\]
Thus it suffices to compare the two formulas for \(3\leq n\leq 15\). By direct computation we find equality precisely when \(3\leq n\leq 8\).
\end{proof}

\end{document}